\documentclass[12pt]{amsart}
\usepackage{graphicx}
\newtheorem{theorem}{Theorem}
\newtheorem{definition}[theorem]{Definition}
\begin{document}

\centerline{\Large \bf On charges equilibrium distribution}
\centerline{\Large \bf on conductor}
\vskip 0.5 cm
\centerline{\it Ashot Vagharshakyan}
\vskip 0.3cm
\centerline{\it Institute of Mathematics, National 
Academy of Sciences (Armenia)}
\vskip 0.3cm
{\it Abstract. In this paper, new model for charges
distribution is discussed. It is proved that a small charge inside 
a conductor, no force acts. charge, which is located on the border 
resistant to small perturbations.
The minimum energy of a photon, which could push the electron
from the conductor is estimated.}
\section{\bf Introduction}
	In this paper we propose a new model of the charge distribution.
We give a new definition of potential energy.
The equilibrium state is characterized by a minimum value of the 
potential energy.
Existence and uniqueness of the equilibrium distribution, 
we have proved.
Shown that the potential function for the equilibrium distribution 
is constant inside a conductor. Consequently, on a small charge, 
inside a conductor do not acts force. This result is explained 
wellknown phenomenon in M. Faraday Cage.
It is proved that electrons placed on the border, are stable to
a small perturbation. It follows from this that the charges on the border,
Can not leave the conductor.
Nevertheless, large perturbation may be cause of electron to leave the 
conductor. The photon energy absorbed by an electron, the electron 
energy increases. The energy that can rid of an electron from the 
conductor is estimated. It is a typical situation for the 
photoelectric effect.

\section{\bf Auxiliary information}

	For an arbitrary function 
\begin{equation*}
\varphi(\vec{x}),\,\,\vec {x}\in R^3,
\end{equation*}
denote the nearest point to the set through
$\{\vec {x};\,\,\varphi(\vec{x})=0\}$.
Infinitely differentiable functions 
$\varphi(\vec{x})$
belong to the space 
$\Im$
if it has compact support. The sequence 
$\varphi_n(\vec{x}),\,\,\,n=1,2, \dots$
goes to zero, if it has all derivatives tend uniformly to zero. 

The family of linear and continuous functionals, defined on 
$\Im$ 
are "generalized functions" and they are denoted through
$\Im^*$.
An arbitrary integrable function 
$f(\vec{x}),\,\, \vec{x}\in R^3$
generates a generalized function acting by the formula
\begin{equation*}
l_f(\varphi)=\int_{R^3}f(\vec{x})\varphi(\vec{x})d\vec{x}.
\end{equation*}
Partial derivative of a generalized function is defined  
\begin{equation*}
\frac{\partial l}{\partial x}(\varphi)=
-l\left(\frac{\partial \varphi}{\partial x}\right),\,\,\varphi\in\Im.
\end{equation*}
Analogously define the derivative of other variables.
For an arbitrary generalized function 
$l\in \Im^*$ 
is defined by 
$supp(l)$
the smallest close set such that for an arbitrary 
$\varphi\in \Im$
such that
\begin{equation*}
supp(\varphi) \cap supp(l)=\emptyset
\end{equation*}
we have
$l(\varphi)=0$.
Let us note that if 
$supp(l)$ 
is compact set, then in the natural way, it is possible to define 
$l(\varphi)$ 
for an arbitrary infinitely differentiable function for which 
$supp(\varphi)$ 
is not necessary to be compact.   
\begin{definition}
A function 
$f(\vec{x})\in L_2(R^3)$
belongs to the class of Dirichlet 
$\bf{D}$
if the generalized function 
$l_f$ 
has all partial derivatives
\begin{equation*}
\frac{\partial l}{\partial x},\,\,\,
\frac{\partial l}{\partial x},\,\,\,
\frac{\partial l}{\partial x},
\end{equation*}
which are generalized by functions from 
$L_2(R^3)$.
	The norm in Dirichlet class is defined by
\begin{equation*}
\|f\|_{\bf{D}}^2=\int_{R^3}\left(\left|\frac{\partial l}{\partial x}\right|^2+
\left|\frac{\partial l}{\partial x}\right|^2+
\left|\frac{\partial l}{\partial x}\right|^2\right)d\vec{x}\,\,.
\end{equation*}
\end{definition}
A generalized function
$l\in \Im^*$
belongs to the class
$\bf{D}^*$
if there is a constant
$M <\infty$
that for each
$\varphi \in \bf{D}$
the inequality
\begin{equation*}
|l(\varphi)|\leq M\|\varphi\|_{\bf{D}}.
\end{equation*}
valid. The minimum value
$M$
for which this inequality valid denoted by
$\|l\|_{\bf{D}^*}$.
Note that the known function of Dirac
$\delta$
does not belong
$\bf{D}^*$
\begin{definition}
A generalized function
$l_f\in\bf{D}^*$
has a density at the point
$\vec{x}$
when
\begin{equation*}
\lim_{r\rightarrow 0+}\frac{1}{|B(\vec{x},\,r)|}\int_{B(\vec{x},\,r)}f(\vec{y})dv.
\end{equation*}
\end{definition}
This definition is correct for the generalized functions of the form
$l_f$
We have 
\begin{theorem}
If the generalized function
$l_f\in\bf{D}^*$
has density at the point
$\vec{x}$
then
\begin{equation*}
\lim_{r\rightarrow 0+}\frac{1}{|B(\vec{x},\,r)|}\int_{B(\vec{x},\,r)}f(\vec{y})dv=
\end{equation*}
\begin{equation*}
=\lim_{r\rightarrow 0+}\frac{4}{|B(\vec{x},\,r)|}
\int_{B(\vec{x},\,r)}\left(1-\frac{\|\vec{x}-\vec{y}\|}{r}\right)f(\vec{y})dv.
\end{equation*}
\end{theorem}
In general case introduce the next definition.
\begin{definition}
For an arbitrary generalized function
$l\in\bf{D}^*$
we say that it has density at the point
$\vec{x}$
if the limit
\begin{equation*}
\lim_{r\rightarrow 0+}\frac{4}{|B(\vec{x},\,r)|}l\left(T_{\vec{x}}\right)
\end{equation*}
exists, where
\begin{equation*}
T_{\vec{x}}(\vec{y})=\left(1-\frac{\|\vec{x}-\vec{y}\|}{r}\right)^+,
\,\,\vec{y}\in B(\vec{x},\,r).
\end{equation*}
\end{definition}

\section{Potential function}

As a result, by the influence of neighboring placed atoms, some 
electrons leave atom. Those electrons create a cloud of 
"free" electrons. Family of positive ions and negative 
electrons, are well - balanced. Tat is why, conductor seems neutral. 
The most suitable mathematical tool to describe 
the distribution of charges, is the finite measure.
Go ahead, the charge distribution, we will describe by
generalized function
$l\in \Im^*$.
In this point of view some problems arise. For example, 
the charge, concentrated on the given set, for arbitrary subset
has not a sense. Naturally, the question arises on the 
necessity to introduce generalized functions. 
In example 1. it is proved, that for the equilibrium 
distribution, we must to use generalized functions.
\begin{definition}
When
$F\subset R^3$ 
is a compact set and 
$q$ 
is a real number. denote by
$Ch(F,q)$ 
the set all generalized functions
$l\in \bf{D}^*$
with the support 
$supp(l)\subseteq F$ and $l(1)=q$. 
\end{definition}
\begin{definition}
The capacity of a bounded set 
$F\subset R^3$
is defined by the formula 
\begin{equation*}
C(F)=\frac{1}{4\pi}\int_{R^3\setminus F}\|\nabla U\|^2dv,
\end{equation*}
where 
$U(\vec{x})\in \bf{D}$
is harmonic out of the set 
$F$,
and tends to zero at infinity and 
$U(\vec{x})=1,\,\,\vec{x}\in F$.
\end{definition}
For the ball we have
\begin{equation*}
C(B(\vec{0},\,r))=r.
\end{equation*}
\begin{definition}
Denote by
$k=k(F)>0$ 
the greatest number, for which the following inequality valid
\begin{equation*}
k^2(F)\int_{F}\|f\|^2dv\leq \int_{R^3}\|\nabla f\|^2dv,
\end{equation*}
for each 
$f(\vec{x})\in \bf{D}$.
\end{definition}
\begin{theorem}
Assume
$F=B(\vec{0},\,r)$. 
The estimates
\begin{equation*}
\frac{3}{4r\sqrt{2\pi}}\leq k\left(F\right)\leq\sqrt{\frac{C(F)}{|F|}}
\end{equation*}
valid.
\end{theorem}
\begin{proof}
As $f\in \bf{D}$ and it satisfies the conditions
$f(\vec{x})\equiv 1,\,\,\vec{x}\in F$ 
and it is harmonic function on the set
$R^3\setminus F$. 
Hence we have 
\begin{equation*}
k^2(F)|F|=k^2(F)\int_{F}|f(\vec{x})|^2dv\leq \frac{1}{4\pi}\int_{R^3}|\nabla f(\vec{x})|^2dv=
\end{equation*}
\begin{equation*}
=\frac{1}{4\pi}\int_{R^3\setminus F}|\nabla f(\vec{x})|^2dv=C(F)
\end{equation*}
Consequently, 
\begin{equation*}
k^2(F)\leq \frac{C(F)}{|F|}.
\end{equation*}
Now let us prove the lower bound.
First we give some preliminary results. For an arbitrary function
$f(\vec{x}),\,\, \vec{x}\in R^3,$ 
we define the new function
\begin{equation*}
f^*(\vec{x})=\int_0^{\infty}\chi(\vec{x},t)dt,
\end{equation*}
where 
$\chi(\vec{x},t)=1$ 
if 
$\|\vec{x})\|<R$ 
and
$\chi(\vec{x},t)=0$ if $\|\vec{x})\|\geq R$.
Here 
\begin{equation*}
R^3=\frac{3}{4\pi}\left|\{\vec{x};\,\,f(\vec{x})\leq t\}\right|.
\end{equation*}
It is known that
\begin{equation*}
\int_{R^3}|f(\vec{x})|^2dv=\int_{R^3}|f^*(\vec{x})|^2dv.
\end{equation*}
The inequality
\begin{equation*}
\int_{R^3}\|\nabla f^*(\vec{x})\|^2dv\leq \int_{R^3}\|\nabla f(\vec{x})\|^2dv.
\end{equation*}
valid. So, it is enough to prove, the requared inequality 
for the functions satisfying the condition 
$f(\vec{x})=f^*(\vec{x})$.
For 
$E=B(\vec{x},r)$ 
we have
\begin{equation*}
\int_E|f|^2dv=
\int_0^r\int_0^{2\pi}\int_0^{\pi}|f(\rho)|^2\rho^2sin(\theta) d\rho d\theta d\varphi=
\end{equation*}
\begin{equation*}
=4\pi \int_0^r\left|f(r)-\int_{\rho}^rf'(x)dx\right|^2\rho^2d\rho\leq
\end{equation*}
\begin{equation*}
\leq 8\pi \left(\frac{r^3}{3}|f(r)|^2+
\int_0^r\left(\int_{\rho}^r|f'(x)|dx\right)^2\rho^2d\rho\right).
\end{equation*}
Of Hardy inequality
\begin{equation*}
\int_0^r\left(\int_{\rho}^r|f'(x)|dx\right)^2\rho^2d\rho\leq \frac{4}{9}\int_0^r|f'(x)|^2dx,
\end{equation*}
we have
\begin{equation*}
\int_E|f|^2dv\leq \frac{2r^2}{3}\int_{R^3\setminus E}|\nabla f|^2dv+
\frac{8r^2}{9}\int_{E}|\nabla f|^2dv\leq \frac{8r^2}{9}\int_{R^3}|\nabla f|^2dv. 
\end{equation*}
Consequently,
\begin{equation*}
\frac{9}{32\pi r^2}\int_{E}|f|^2dv\leq \frac{1}{4\pi}\int_{R^3}|\nabla f|^2dv.
\end{equation*}
\end{proof}
\begin{definition}
For arbitrary functions 
$f,\,g\in \bf{D}$ 
let us denote through
\begin{equation*}
(f,\,g)_k=\frac{1}{4\pi}\int_{R^3}(\nabla f,\,\nabla g)dv-k^2\int_Ef\cdot gdv.
\end{equation*}
\end{definition}
\begin{definition}
The unique function 
$U^l_k(\vec{x})\in \bf{D}$ 
is $k$  - potential function for the charge distribution 
$l\in \bf{D}^*$ 
if for an arbitrary  
$\varphi \in \bf{D}$ 
we have
\begin{equation*}
l(\varphi)=(\varphi,\,U^l_k).
\end{equation*}
\end{definition}
The existence and uniqueness of potential function follows of M. Riesz theorem.
\begin{definition}
The unite vector 
$\bar{n}$ 
is orthogonal to 
$E$ 
at the point 
$\vec{x}_0\in E$ 
if
\begin{equation*}
\lim_{E\ni \vec{x}\rightarrow \vec{x}_0}
\frac{(\vec{x}-\vec{x}_0,\,\bar{n})}{\|\vec{x}-\vec{x}_0\|}=0.
\end{equation*}
\end{definition}
\begin{definition}
The unit vector  
$\bar{n}_+$ 
is an outer normal 
$E$ 
at the point 
$\vec{x}_0\in E$ 
if it is orthogonal to 
$E$ 
and there is a number 
$0<r$ 
such that
\begin{equation*}
B(\vec{x}+r \bar{n}_+,\,r)\cap E=\emptyset.
\end{equation*}
\end{definition}
If
$\bar{n}$ 
is a unit vector and 
$f(\vec{x}),\,\,\vec{x}\in R^3,$
is a function denote Through
\begin{equation*}
\frac{\partial f(\vec{x})}{\partial \bar{n}}=\lim_{t\rightarrow +0}
\frac{f(\vec{x}+t\bar{n})-f(\vec{x})}{t}
\end{equation*}
if the limit exists. Let us note that in this definition only 
the values of the function on 
$\{\vec{x}+t\bar{n}; 0<t\}$, 
we use. If the function 
$f$ 
is differentiable at the point 
$\vec{x}$ 
then
\begin{equation*}
\frac{\partial f(\vec{x})}{\partial \bar{m}}=
-\frac{\partial f(\vec{x})}{\partial \bar{n}},
\end{equation*}
where 
$\vec{m}=-\vec{n}$. 
In this statement the differentiability plays essential since for the function  
$f(\vec{x})=\|\vec{x}\|$ 
we have
\begin{equation*}
\frac{\partial f(\vec{0})}{\partial \bar{m}}=
\frac{\partial f(\vec{0})}{\partial \bar{n}}=1.
\end{equation*}
\begin{theorem}
If 
$E$ 
is a compact set with smooth boundary 
$\partial E$. 
If
$l\in \bf{D}^*$ 
has support 
$supp(l)\subseteq E$. 
Let the potential function 
$U_k^l(\vec{x})\in \bf{D}$ 
almost everywhere has all derivatives of the second order at the points 
$\vec{x}\in R^3\setminus \partial E$, 
wtich are integrable functions in 
$R^3$.
Then for arabitrary 
$\varphi \in \Im$ 
we have
\begin{equation*}
l(\varphi)=-k^2\int_E\varphi U_k^ldv-
\int_{\partial E}\varphi \left(\frac{\partial U_k^l}{\partial \vec{n}_+}+
\frac{\partial U_k^l}{\partial \vec{n}_-} \right)ds-\frac{1}{4\pi}\int_{R^3}\varphi U_k^ldv,
\end{equation*}
where $\vec{n}_+$ be outer normal to $E$ and $\vec{n}_-$ be inner normal to $E$.
\end{theorem}
\begin{proof}
\begin{equation*}
\Omega_+(t)=\Omega\setminus \left(E\cup \{\vec{x}+s\vec{n}_+(\vec{x});
\,\,0<s\leq t,\,\vec{x}\in \partial E\}\right)
\end{equation*}
and
\begin{equation*}
\Omega_-(t)=\left(\Omega \cap E\right)\setminus 
\left(\{\vec{x}+s\vec{n}_-(\vec{x});\,\,0<s\leq t,\,\vec{x}\in \partial E\}\right)
\end{equation*}
Let 
$\vec{m}_+(\vec{y}$ 
be the unit outer normal to 
$\Omega_+(t)$ 
at the point
$\vec{y}\in \partial \Omega_+(t)$ 
and 
$\vec{m}_-(\vec{y}$ 
be the unit inner normal to 
$\Omega_-(t)$ 
at the point
$\vec{y}\in \partial \Omega_+(t)$. 
Let us note that for a small value of 
$t$ 
at each point
\begin{equation*}
\vec{x}+t\vec{n}_+(\vec{x}),\,\,\,\vec{x}\in \partial \Omega_+(t)
\end{equation*}
we have 
$\vec{m}_+(\vec{y})=-\vec{n}_+(\vec{y})$ 
and at each point
\begin{equation*}
\vec{x}+t\vec{n}_-(\vec{x}),\,\,\,\vec{x}\in \partial \Omega_-(t)
\end{equation*}
we have 
$\vec{m}_-(\vec{y})=-\vec{n}_-(\vec{y})$.
We have
\begin{equation*}
l(\varphi)=\frac{1}{4\pi}\int_{R^3}\left(\nabla \varphi,
\,\,\nabla U_k^l\right)dv-
k^2\int_{R^3}\varphi U_k^ldv=
\end{equation*}
\begin{equation*}
=-k^2\int_{R^3}\varphi U_k^ldv+
\end{equation*}
\begin{equation*}
=\lim_{t\rightarrow +0}\frac{1}{4\pi}\int_{\Omega_+(t)}\left(\nabla \varphi,
\,\,\nabla U_k^l\right)dv+
\lim_{t\rightarrow +0}\frac{1}{4\pi}\int_{\Omega_-(t)}\left(\nabla \varphi,
\,\,\nabla U_k^l\right)dv
\end{equation*}
By Green's formula we have
\begin{equation*}
\lim_{t\rightarrow +0}\frac{1}{4\pi}\int_{\Omega_+(t)}\left(\nabla \varphi,
\,\,\nabla U_k^l\right)dv=
\end{equation*}
\begin{equation*}
=\lim_{t\rightarrow +0}
\left(\frac{1}{4\pi}\int_{\partial\Omega_+(t)}\varphi\frac{\partial U_k^l}{\partial \vec{n}_+}ds-
\frac{1}{4\pi}\int_{\Omega_+(t)}\varphi \Delta U_k^ldv\right)=
\end{equation*}
\begin{equation*}
=\frac{1}{4\pi}\int_{\partial E}\varphi\frac{\partial U_k^l}{\partial \vec{n}_+}ds-
\frac{1}{4\pi}\int_{\Omega \cap E}\varphi \Delta U_k^ldv
\end{equation*}
Similary we have
\begin{equation*}
\lim_{t\rightarrow +0}\frac{1}{4\pi}\int_{\Omega_-(t)}
\left(\nabla \varphi,\,\,\nabla U_k^l\right)dv=
\frac{1}{4\pi}\int_{\partial E}\varphi\frac{\partial U_k^l}{\partial \vec{n}_-}ds-
\frac{1}{4\pi}\int_{\Omega \cap E}\varphi \Delta U_k^ldv
\end{equation*}
\end{proof}

\section{Equlibrium distribution}

\begin{definition}
If 
$0<k\leq k(E)$. 
Then
$k$ 
- potential energy of distribution 
$l\in \bf{D}^*$ 
is defined by
\begin{equation*}
W_k(l)=l\left(U_k^l\right)=\left(U_k^l,\,U_k^l\right)_k.
\end{equation*}
\end{definition}
\begin{theorem}
If 
$E$ 
is a compact and $q$ - is a real number then the functional 
$W_k(l),\,\,l\in Ch(E,\,q)$
is convex.
\end{theorem}

\begin{proof}
Let 
$l_1,\,l_2\,\in Ch(E,q)$. 
We need to prove the inequality
\begin{equation*}
W_k\left(\frac{l_1+l_2}{2}\right) \leq \frac{W_k(l_1)+W_k(l_2)}{2}.
\end{equation*}
This inequality is equivalent to the following
\begin{equation*}
\left(U_k^{l_1+l_2},\,U_k^{l_1+l_2}\right)_k\leq 
2\left(U_k^{l_1},\,U_k^{l_1}\right)_k+2\left(U_k^{l_2},\,U_k^{l_2}\right)_k.
\end{equation*}
After elementary transformation we come to
\begin{equation*}
0\leq \left(U_k^{l_1-l_2},\,U_k^{l_1-l_2}\right)_k.
\end{equation*}
\end{proof}

\begin{definition}
Let $E$ be a compact set and $q$ be a real number. If for $l_k\in Ch(E,\,q)$
we have 
\begin{equation*}
W_k(l_k)=\inf \left\{W_k(l);\,\,\,l\in Ch(E,\,q)\right\}
\end{equation*}
The distribution $l_k$ is $k$ equilibrium distribution in $Ch(E,\,q)$.
\end{definition}

\begin{theorem}
If $E$ is a compact set and $q$ is a real number. Then there is a unique 
$k$ equilibrium distribution in  $Ch(E,\,q)$.
\end{theorem}
\begin{proof}
Thunks of the known result from functional analysis it is enough to note that the set 
$Ch(E,\,q)$ is close end convex.
\end{proof}
\begin{theorem}
If $E$ is a compact set and 
$q$ 
is a real number then the potential function 
$U_k^{l_k}$  
is constant inside of conductor.
\end{theorem}
\begin{proof}
Let 
$B(\vec{x}_1, r)$ 
and 
$B(\vec{x}_2, r)$ 
be nonintersecting balls placed inside of 
$E$. 
Denote
\begin{equation*}
l_r^1(\varphi)=\frac{3}{4\pi r^3}\int_{B(\vec{x}_1,\,r)}\varphi dv,\,\,\,
l_r^2(\varphi)=\frac{3}{4\pi r^3}\int_{B(\vec{x}_2,\,r)}\varphi dv.
\end{equation*}
By definition we have
\begin{equation*}
W_k(l_k)\leq W_k(l_k+tl_r^1-tl_r^2). 
\end{equation*}
From this inequality for all 
$0<t$ 
it follows
\begin{equation*}
l_r^1\left(U_k^{l_1}\right)\geq l_r^2\left(U_k^{l_1}\right).
\end{equation*}
Passing to the limit when
$r\rightarrow +0$ 
we get
\begin{equation*}
U_k^{l_1}(\vec{x}_1)\geq U_k^{l_1}(\vec{x}_2).
\end{equation*}
\end{proof}
{\bf Example 1.} Let us give an example where we cannot avoid the generalized functions. Let
\begin{equation*}
E=\partial B(\vec{0},\,1)\cup \bar{B}(\vec{0},\,r_1) 
\left(\bigcup_{n=1}^{\infty}\bar{B}(\vec{0},\,r_{2n+1})\setminus B(\vec{0},\,r_{2n})\right),
\end{equation*}
where 
$0<r_1<r_2<\dots <1$ 
and 
$\lim_{n\to \infty}r_n=1$. 
Let us put a charge 
$q$ 
on the ball 
$\bar{B}(\vec{0},\,r_1)$. 
In the equilibrium state, the charge 
$q_n$ 
inducted on each sphere 
$\partial B(\vec{0},\,r_n)$. 
We have 
\begin{equation*}
q_{2n}+q_{2n+1}=0,\,\,\,n=1,2,\dots
\end{equation*}
The potential function can be presented in the form
\begin{equation*}
U(\vec{x})=\sum_{n=1}^{\infty}q_n\min\left(\frac{1}{r_n},\,\,\frac{1}{\|\vec{x}\|}\right).
\end{equation*}
Consequently, at the points 
\begin{equation*}
\vec{x}\in \bar{B}(\vec{0},\,r_{2n+1})\setminus B(\vec{0},\,r_{2n})
\end{equation*}
we have
\begin{equation*}
U(\vec{x})=\frac{1}{\|\vec{x}\|}\sum_{k=1}^{2n}q_n+
\sum_{k=n}^{\infty}\left(\frac{q_{2k+1}}{r_{2k+1}}+\frac{q_{2k}}{r_{2k}}\right).
\end{equation*}
Since the potential function is constant on each connected part, so
\begin{equation*}
\sum_{k=1}^{2n}q_k,\,\,\,n=1,2,\dots
\end{equation*}
Consequently, we have
\begin{equation*}
q_n=(-1)^{n+1}q_1.
\end{equation*}
From this follows, that the equilibrium distribution can not be a finite measure.
	Since the potential function is constant inside of conductor
\begin{equation*}
U_k^{l_k}(\vec{x})=A,\,\, \vec{x}\in E
\end{equation*}
so 
\begin{equation*}
l_k(\varphi)=-k^2\int_E \varphi U_k^{l_k}dv-
\int_{\partial E}\varphi\left(\frac{\partial U_k^{l_k}}{\partial\vec{n}_+}+
\frac{\partial U_k^{l_k}}{\partial\vec{n}_-}\right)ds-\frac{1}{4\pi}\int_E \varphi \Delta U_k^{l_k}dv=
\end{equation*}
\begin{equation*}
=-A\int_E \varphi dv-
\frac{1}{4\pi}\int_{\partial E}\varphi\frac{\partial U_k^{l_k}}{\partial\vec{n}_+}ds.
\end{equation*}
We have negative charges 
$l(1)=q$, 
so
\begin{equation*}
l_k(1)=q=-k^2A|E|+AC(E).
\end{equation*}
From this equation we get 
\begin{equation*}
A=\frac{q}{C(E)-k^2|E|}
\end{equation*}
Inside of the conductor we have only positive ions
\begin{equation*}
Q=-\frac{qk^2|E|}{C(E)-k^2|E|}
\end{equation*}
On the boundary we have only the negative electrons equal
\begin{equation*}
\hat{q}=\frac{qC(E)}{C(E)-k^2|E|}
\end{equation*}
Let us note, that if $k=0$ then all negative charges, which we 
bring on conductor, migrate to the conductor boundary. However, 
in the case $k>0$, some of "free electrons"  migrate to the boundary also.
{\bf Example 2.}
Let 
$E=\bar{B}(\vec{0},\,r)$. 
Since the conductor is  simmetry we have
\begin{equation*}
l_k(\varphi)=\frac{Q}{|B(\vec{0},\,r)|}\int_{B(\vec{0},\,r)}\varphi dv+
\frac{\hat{q}}{|\partial B(\vec{0},\,r)|}\int_{\partial B(\vec{0},\,r)}\varphi ds.
\end{equation*}
\section{forces on conductor}
\begin{definition}
Let 
$l\in \bf{D}^*$ 
and the 
$k$
- potential function 
$U_k^{l}(\vec{x})$ 
be differentiable at the point 
$\vec{x}$. 
Then on a small charge $e$, which is placed at the point  
$\vec{x}$ 
the force is defined by the formula
\begin{equation*}
\bar{F}_k(\vec{x})=
e\lim_{r\rightarrow 0+}\frac{4}{|B(\vec{x},\,r)|}
\int_{B(\vec{x},\,r)}U_k^l(\vec{y})\frac{\vec{x}-\vec{y}}{\|\vec{x}-\vec{y}\|}dv.
\end{equation*}
\end{definition}
Let us note that if 
$l=l_f$ 
and the potential function is differentiable then 
\begin{equation*}
\bar{F}_k(\vec{x})=-e\nabla U_k^{l_k}(\vec{x}).
\end{equation*}
{\bf Example 3.} Let on the ball 
$E=\bar{B}(\vec{0},\,r)$ 
we put a charge 
$q$ 
and this charge is in equilibrium state. On the small charge placed at the point 
$\vec{x}\notin B(\vec{0},\,r)$ 
out of the bll, acts the fource
\begin{equation*} 
\bar{F}_k(\vec{x})= -e\nabla U_k^{l_k}(\vec{x})=\frac{eq\vec{x}}{\|\vec{x}\|^3}.
\end{equation*}
	If the charge $e$ is inside of conductor at the point 
$\vec{x}\in B(\vec{0},\,r)$  
the force acting on that charge, equals zero. 
Let us note, that in the last case, the force corresponding only by electrical interaction equals
\begin{equation*} 
\bar{F}_k(\vec{x})= -e\nabla U_0^{l_k}(\vec{x})=-\frac{6eq\vec{x}}{r(3-4\pi k^2r^2)}.
\end{equation*}
Note that this forces have direction to the center of the ball.	
We assume, that the "free electrons" collide with the ions, 
which appear in conductor as a result of migration some of "free electrons" on the boundary.
Those collides are cause of new forces.
\begin{equation*} 
\bar{F}_k^c(\vec{x})= M\vec{x}.
\end{equation*}
This forces are directed from the centre. If
\begin{equation*} 
M=\frac{2Qk^2}{|B(\vec{0},\,r)|}.
\end{equation*}
this force balanced by the electric force. In the case 
$k>0$ 
the inclusion 
$supp(l_k)\subseteq \partial E$ 
is not valid. However, for the equilibrium distribution 
$l_k$, 
we have
\begin{equation*} 
supp(\vec{F}_k)\subseteq \partial E.
\end{equation*}
Since, potential function is constant, inside of conductor.
In 1836 year M. Faraday by experiment proved, that inside of conductor there is no forces. 
The same result in 1755 year observed BY. Franklin. M. Faraday, formulate the hypothesis, 
that all charges bring on conductor migrate to the boundary. 
Here we proved the effect about forces, but for the M. Faraday hypothesis we 
give a negative answer.

\section{Photoelectrical effect}

Photoeffect is connected with the emission of electrons on influence of photons.
This phenomenon in 1839 was observed by A. Bekkerel in electrolit. Further in 1887 
the same effect was investigated G. Hertz. He observed that lighting by 
ultra-violet rays of the  discharger, then the current increases.  
Physicists, in experiments, discovered the following laws: 

1. the number of emit electrons is proportional to the intensity of photon;

2. the energy, of pushing out electrons, is proportional of frequency of of photons;

3. there is a minimal photon energy, which can bring out electron from conductor.

	The theoretical explanation those laws was given by A. Einshtain in 1905. In this 
theory the photon energy equals $h\nu$, where $h$ is Plank's constant and $\nu$ 
is frequency of photon. 

	A part of falling photons are reflected from the boundary of conductor. Other part 
penetrated into metall and absorbed by electrons.The absorbed photons increased 
the electron energy. This permits electrons to leave metall. In 1921 Millicen 
proved the A. Einshtain laws in experiment.

	Here we will prove that the process of pushing out of electrons is possible to explain.
The essential moment in our model is pair of electrons placed in the distance
$\delta =1.45\cdot 10^{-8}$ m. 
The idea of "electron pair" appeared 1916 year in G. Lewis
paper [7]. Later this idea successfully applied by L. Cooper [4]
in superconductivity.	In equilibrium state part of electrons 
concentrated on the boundary. We assume that those electrons form 
pairs and the line passing through centres of electrons is orthogonal 
to the boundary of conductor.
{\bf Example.} Let 
$B(\vec{0},\,r)$ 
and 
$0<t<1$. 
The equilibrium distribution is
\begin{equation*}
l_k(\varphi)=\frac{Q}{|B(\vec{0},\,r)|}\int_{B(\vec{0},\,r)}\varphi dv+
\end{equation*}
\begin{equation*}
+\frac{\hat{q}t}{|\partial B(\vec{0},\,r)|}
\int_{\partial B(\vec{0},\,r)}\varphi ds+
\frac{\hat{q}(1-t)}{|\partial B(\vec{0},\,r+\delta)|}
\int_{\partial B(\vec{0},\,r+\delta)}\varphi ds.
\end{equation*}
By different reasons, some of "electron  pair" can go to the points
\begin{equation*}
\vec{x},\,\,\,\vec{x}+\delta \vec{n}_+(\vec{x}),\,\,\,\vec{x}\in B(\vec{0},\,R).
\end{equation*}
In the case 
$r<R<r+\delta$ 
on the electron, which is placed at the point 
$\vec{x}$ 
will act a force caused by positive ions inside of conductor which equals
\begin{equation*}
-\frac{qek^2|E|}{R^2(C(E)-k^2|E|)}\frac{\vec{x}}{\|\vec{x}\|}.
\end{equation*}
On inside placed electron act the force caused by electrons placed on the sphere 
$\partial B(\vec{0},\,r)$
which equals
\begin{equation*}
\frac{qek^2C(E)t}{R^2(C(E)-k^2|E|)}\frac{\vec{x}}{\|\vec{x}\|}.
\end{equation*}
Let us note that the electrons placed on the sphere 
$\partial B(\vec{0},\,r+\delta)$
do not act on this electron. Consequently, the total force equals
\begin{equation*}
\frac{qe(tC(E)-k^2|E|}{R^2(C(E)-k^2|E|)}\frac{\vec{x}}{\|\vec{x}\|}.
\end{equation*}
on this pair do not act forces caused by chaotic motion of "free electrons".
On the charge placed at the point 
$\vec{x}+\delta \vec{n}_+(\vec{x})$ 
act the following force
\begin{equation*}
\frac{qe}{(R+\delta)^2}\frac{\vec{x}}{\|\vec{x}\|}.
\end{equation*}
The resulting force on the pair equals
\begin{equation*}
qe\left(\frac{1}{(R+\delta)^2}-\frac{(tC(E)-k^2|E|}{R^2(C(E)-k^2|E|)}\right)
\frac{\vec{x}}{\|\vec{x}\|}.
\end{equation*}
If
\begin{equation*}
tC(E)-k^2|E| < 0
\end{equation*}
The parametre is defined from the equation
\begin{equation*}
\frac{tC(E)-k^2|E|}{C(E)-k^2|E|}
\frac{1}{(r+\delta)^2}+\frac{1}{(r+2\delta)^2}=0
\end{equation*}
or
\begin{equation*}
\frac{3t-4\pi k^2r^2}{3-4\pi k^2r^2}=
-\frac{(r+\delta)^2}{(r+2\delta)^2}
\end{equation*}
This equation mean that one of electron in pair  placed on the sphere
$\partial B(\vec{0},\,r+\delta)$ 
then the total force on the pair is zero. Consequently, for arbitrary 
$r<R<r+\delta$ 
we have
\begin{equation*}
\frac{3t-4\pi k^2r^2}{3-4\pi k^2r^2}\frac{1}{R^2}+\frac{1}{(R+\delta)^2}=
\end{equation*}
\begin{equation*}
=\frac{1}{R^2}\left(\frac{R^2}{(R+\delta)^2}- 
\frac{(r+\delta)^2}{(r+2\delta)^2}\right)<0\,\,.
\end{equation*}
So, the total force is directed to the centre of the ball and it block up the pair
to return to the original place.	If the displacement is bigger
$r+\delta<R$ 
then this pair will push
out the conductor. So, we come to the situation typical for
photoelectric effect. The minimal energy of photon, which can be 
cause of that perturbation, permits the estimation
\begin{equation*}
E>-qe\int_r^{r+\delta}\left(\frac{1}{(R+\delta)^2}-
\frac{(r+\delta)^2}{(r+2\delta)^2R^2}\right)dR\sim \frac{qe}{r^4}
\end{equation*} 
\section{Conclusion}

	Using brownian motions of "free electrons" we explain
that the boundary electons do not leave conductor.
We get the theoretical expanation of photoelectric effect too.

\end{document}